\documentclass[12pt]{amsart}
\usepackage{amsmath}
\usepackage{amssymb}
\usepackage[all]{xy}
\usepackage{longtable}
\usepackage[osf,sc]{mathpazo}
\usepackage{calrsfs}
\usepackage{color}
\usepackage{multirow,bigdelim}
\usepackage{stmaryrd}
\usepackage{url}

\newtheorem{theorem}[equation]{Theorem}
\newtheorem{lemma}[equation]{Lemma}
\newtheorem{proposition}[equation]{Proposition}
\newtheorem{corollary}[equation]{Corollary}
\newtheorem{conjecture}[equation]{Conjecture}
\newtheorem{definition}[equation]{Definition}
\newtheorem{question}[equation]{Question}

\newtheorem{example}[equation]{Example}

\theoremstyle{remark}
\newtheorem{remark}[equation]{Remark}

\numberwithin{equation}{subsection}

\allowdisplaybreaks[1]

\newcommand{\FF}{\mathbb{F}}

\newcommand{\GG}{\mathbb{G}}

\newcommand{\CC}{\mathbb{C}}

\newcommand{\NN}{\mathbb{N}}

\newcommand{\bu}{\mathbf{u}}
\newcommand{\bv}{\mathbf{v}}

\newcommand{\bs}{\mathbf{s}}

\DeclareMathAlphabet{\matheur}{U}{eur}{m}{n}

\newcommand{\fs}{\mathfrak{s}}
\newcommand{\fm}{\mathfrak{m}}

 \DeclareMathOperator{\Lie}{Lie}
 
\DeclareMathOperator{\Mat}{Mat}

 \DeclareMathOperator{\wt}{wt}
  
\DeclareMathOperator{\Li}{Li}    

\DeclareMathOperator{\ad}{ad}

\DeclareMathOperator{\dep}{dep}
\DeclareMathOperator{\height}{ht}
\DeclareMathOperator{\ord}{ord}

\newcommand{\ok}{\overline{k}}

\newcommand{\tr}{\mathrm{tr}}

\begin{document}

\title[Integrality of $v$-adic multiple zeta values]{{\large{I\MakeLowercase{ntegrality of $v$-adic multiple zeta values} }}}

\author{Yen-Tsung Chen}
\address{Department of Mathematics, National Tsing Hua University, Hsinchu City 30042, Taiwan
  R.O.C.}

\email{s107021901@m107.nthu.edu.tw}

\thanks{The author was partially supported by Prof. C.-Y. Chang's MOST Grant 107-2628-M-007-002-MY4}

\subjclass[2010]{Primary 11R58, 11M32}

\date{Jan. 6, 2020}

\begin{abstract} 
    In this article, we prove the integrality of $v$-adic multiple zeta values (MZVs). For any index $\mathfrak{s}\in\mathbb{N}^r$ and finite place $v\in A:=\mathbb{F}_q[\theta]$, Chang and Mishiba introduced the notion of the $v$-adic MZVs $\zeta_A(\mathfrak{s})_v$, which is a function field analogue of Furusho's $p$-adic MZVs. By estimating the $v$-adic valuation of $\zeta_A(\mathfrak{s})_v$, we show that $\zeta_A(\mathfrak{s})_v$ is a $v$-adic integer for almost all $v$. This result can be viewed as a function field analogue of the integrality of $p$-adic MZVs, which was proved by Akagi-Hirose-Yasuda and Chatzistamatiou.
\end{abstract}

\keywords{}

\maketitle
\tableofcontents
\section{Introduction}
\subsection{Classical multiple zeta values}
    Real-valued multiple zeta values, abbreviated as MZVs, are real numbers defined by $$\zeta(s_1,\dots,s_r):=\sum_{n_1>\cdots>n_r\geq 1}\frac{1}{n_1^{s_1}\cdots n_r^{s_r}}\in\mathbb{R}^\times,$$ where $\fs:=(s_1,\dots,s_r)\in \mathbb{N}^r$ and $s_1\geq 2$. Here $\dep(\fs):=r$ is called the \emph{depth}, $\wt(\fs):=\sum_{i=1}^{r}s_i$ is called the \emph{weight} and $\height(\fs):=$ the cardinality of $\{i\mid s_i\neq 1\}$ is called the \emph{height}. Interesting properties of MZVs have been established in recent years, but there remain mysteries. For example, despite the many relations between MZVs which have been discovered, the exact $\mathbb{Q}$-linear structure of MZVs is still unclear. We refer the reader to \cite{BGF18,IKZ06,K19,Zh16} for more details.
    
    Real-valued MZVs come in many variants, we first briefly review the $p$-adic MZVs introduced by Furusho in \cite{F04}. Consider the one-variable multiple polylogarithm $$\Li_{(s_1,\dots,s_r)}(z):=\sum_{n_1>n_2>\cdots>n_r\geq 1}\frac{z^{n_1}}{n_1^{s_1}\cdots n_r^{s_r}},$$ where $(s_1,\dots,s_r)\in \mathbb{N}^r$ and $s_1\geq 2$. We have $$\zeta(s_1,\dots,s_r)=\Li_{(s_1,\dots,s_r)}(z)\mid_{z=1}.$$ We write $\Li_{(s_1,\dots,s_r)}(z)_p$ for the $p$-adic function defined by the same series on $\mathbb{C}_p$. Then $\Li_{(s_1,\dots,s_r)}(z)_p$ converges on the open unit disk centered at $0$. Thus, in the non-archimedean context, it does not make sense to take the limit $z \to 1$. But Furusho \cite{F04} applied Coleman integration \cite{Col82} to $p$-adically analytically continue $\Li_{(s_1,\dots,s_r)}(z)_p$ to $\mathbb{C}_p\setminus\{1\}$, and then took a certain limit $z \to 1$ to define the $p$-adic MZV $\zeta(s_1,\dots,s_r)_p$. These $p$-adic MZVs have features in common with real-valued MZVs; for example, it is shown in \cite{FJ07} that $p$-adic MZVs satisfy the same regularized double shuffle relations as real-valued MZVs satisfy \cite{IKZ06}.
    
    Furusho mentioned the question of whether his $p$-adic MZVs have integral values, that is, whether $\zeta(s_1,\dots,s_r)_p\in\mathbb{Z}_p$ for all primes $p$ and $(s_1,\dots,s_r)\in\mathbb{N}^r$. Recently, this question was answered by Akagi-Hirose-Yasuda and Chatzistamatiou.
    \begin{theorem}[\cite{AHY,Cha17}]\label{p-adic integrality}
        Every $p$-adic MZV is a $p$-adic integer. Moreover, fix an index $(s_1\dots,s_r)\in\mathbb{N}^r$; then for all but finitely many primes $p$, the $p$-adic valuation of $p$-adic MZVs is greater than the weight $\sum_{i=1}^{r}s_i$.
    \end{theorem}
    
    Now we describe an application of Theorem \ref{p-adic integrality}, given in \cite{AHY}. Consider the $\mathbb{Q}$-algebra $\mathcal{A}:=(\prod_p\mathbb{Z}/p\mathbb{Z})\otimes_\mathbb{Z}\mathbb{Q}$ where $p$ runs over all primes $p$. Kaneko and Zagier defined \emph{finite multiple zeta values} (abbreviated as FMZVs) by $$\zeta_\mathcal{A}(s_1,\dots,s_r):=(\zeta_\mathcal{A}(s_1\dots,s_r)_p)_p\in\mathcal{A},$$ where the $p$-component $\zeta_\mathcal{A}(s_1,\dots,s_r)_p$ is defined by $$\sum_{p>n_1>n_2>\cdots>n_r\geq 1}\frac{1}{n_1^{s_1}\dots n_r^{s_r}}\mbox{ mod }p.$$ Let $w\in\mathbb{N}$, $Z_{w,\mathcal{A}}$ be the $\mathbb{Q}$-vector space generated by all FMZVs of weight $w$, and $Z_w$ be the $\mathbb{Q}$-vector space generated by all MZVs of weight $w$. Kaneko and Zagier made the following conjecture:
    \begin{conjecture}[Kaneko-Zagier]\label{Zagier's conjecture}
        The following identity holds: $$\frac{1}{1-X^2-X^3}=\sum_{w\geq 0}(\dim_\mathbb{Q}Z_w)X^w.$$ Moreover, if we set $d_w:=\dim_\mathbb{Q}Z_w$, then for each $w\geq 2$ we have $$\dim_\mathbb{Q}Z_{w,\mathcal{A}}=d_w-d_{w-2}.$$
    \end{conjecture}
    Akagi-Hirose-Yasuda combine the integrality of $p$-adic MZVs and the special case of Jarossay's result \cite[(0.3.8)]{J18} to give the upper bound of above conjecture for FMZVs.
    \begin{theorem}[\cite{AHY}]\label{upperbound of FMZV}
        For each integer $w\geq 2$, we have $\dim_\mathbb{Q}Z_{w,\mathcal{A}}\leq d_w-d_{w-2}$.
    \end{theorem}
    It is natural to ask whether or not the same phenomena occur in the characteristic $p$ analogue. The main purpose of this article is to study the function field analogue of Theorem~\ref{p-adic integrality} and try to build up a suitable framework for future study of the function field analogue of Theorem~\ref{upperbound of FMZV}.
    
\subsection{Multiple zeta values in positive characteristic}
    The function field analogue of real-valued MZVs were defined by Thakur in \cite{T04}, generalizing Carlitz zeta values \cite{Ca35}. Let $A:=\mathbb{F}_q[\theta]$, $k:=\mathbb{F}_q(\theta)$, and define $k_\infty$ to be the completion of $k$ at the infinite place denote by $\infty$. For any index $(s_1,\dots,s_r)\in\mathbb{N}^r$, the $\infty$-adic MZV is defined by the series $$\zeta_A(s_1,\dots,s_r):=\sum\frac{1}{a_1^{s_1}\cdots a_r^{s_r}}\in k_\infty,$$ where $(a_1,\dots,a_r)\in A^r$ with $a_i$ monic and $\deg_\theta a_i$ strictly decreasing. In \cite{T09} Thakur showed that $\zeta_A(s_1,\dots,s_r)$ is non-vanishing for all indices $(s_1,\dots,s_r)\in\mathbb{N}^r$. For the transcendence problem, Yu \cite{Yu91} proved that all single zeta values $\zeta_A(s_1)$ are transcendental and Chang \cite{C14} proved the transcendence of $\zeta_A(s_1,\dots,s_r)$ for all indices $(s_1,\dots,s_r)$. This problem remains open for real-valued MZVs in classical transcendence theory. Some interesting features of real-valued MZVs remain valid for their characteristic $p$ counterpart. For example, Terasoma \cite{Te02} and Goncharov \cite{Gon02} showed that real-valued MZVs are periods of mixed Tate motives, and Anderson-Thakur \cite{AT09} showed that $\infty$-adic MZVs appear as periods of $t$-motives.
    
    Inspired by Furusho's definition of $p$-adic MZVs, Chang and Mishiba considered the Carlitz multiple star polylogarithms (abbreviated as CMSPLs) $$\Li^\star_{(s_1,\dots,s_r)}(z_1,\dots,z_r):=\underset{i_1\geq\cdots\geq i_r\geq 0}{\sum}\frac{z_1^{q^{i_1}}\cdots z_r^{q^{i_r}}}{L_{i_1}^{s_1}\cdots L_{i_r}^{s_r}},$$ where $(s_1,\dots,s_r)\in\mathbb{N}^r$, $L_0:=1$ and $L_i:=(\theta-\theta^q)\cdots(\theta-\theta^{q^i})$ for $i\geq 1$. Chang and Mishiba \cite[Thm.~5.2.5]{CM19b} proved that $\infty$-adic MZVs can be written as a $k$-linear combinations of CMSPLs at some precise integral points with explicit coefficients. Let $v$ be a fixed finite place of $k$ and regard $\Li_{(s_1,\dots,s_r)}^\star(z_1,\dots,z_r)_v$ as a $v$-adic function defined by the same series on $\mathbb{C}_v^r$ where $\mathbb{C}_v$ is the completion of an algebraic closure of the completion of $k$ at $v$. Then the series converges on the open unit disk centered at $0$. Chang and Mishiba \cite[Prop.~4.1.1]{CM19a} use the logarithmic interpretation to do the analytic continuation of CMSPLs $v$-adically. Moreover, they defined $v$-adic MZVs in \cite[Def.~6.1.1]{CM19b} by using the same $k$-linear combinations of CMSPLs, viewed $v$-adically, and they used Yu's sub-$t$-module theory \cite{Yu97} to prove that $v$-adic MZVs satisfy the same $k$-linear relations as the corresponding $\infty$-adic MZVs do. This result can be viewed as a positive answer of the function field analogue of Furusho's conjecture, which predicted that the $p$-adic MZVs satisfy the same $\mathbb{Q}$-linear relations that the corresponding real-valued MZVs satisfy, and this conjecture remains open in the classical theory. These $v$-adic MZVs are the main objects in the study of this paper.

\subsection{Outline and main result}
    In Section $2$, we fix our notation and setting, then briefly review the terminology of $t$-modules as in \cite{A86}.
    
    In the section $3$, we first follow \cite{CM19a} closely to review the logarithmic interpretation and $v$-adic analytic continuation of CMSPLs. Based on the logarithmic interpretation and the functional equation for logarithms of $t$-modules, it is reasonable to establish certain functional equation for CMSPLs arising from logarithm of $t$-modules. But the difficulty here is that the logarithmic interpretation for CMSPLs rely on specific evaluation for the logarithm of $t$-modules, and this situation may not hold after analytic continuation. We overcome this issue in Section $3$. As an application, we study the $v$-adic valuation of $v$-adic CMSPLs and prove that the $k$-vector space generated by the values of $v$-adic CMSPLs forms an algebra. The latter result is clear for $\infty$-adic CMSPLs by stuffle relations, but the stuffle relations are still unclear for the $v$-adic analytic continuation of CMSPLs. Here, we prove it by the functional equation to avoid the use of stuffle relations.
    
    In Section $4$, we first follow \cite{CM19b} closely to construct $v$-adic MZVs via $v$-adic CMSPLs. Then we explain how to arrive our main result via the properties of $v$-adic CMSPLs obtained in Section $3$. Our main result, stated as Theorem~\ref{Main result} later, is the following:
    \begin{theorem}
        Let $\fs=(s_1,\cdots,s_r)\in\mathbb{N}^r$ and $q_v$ be the cardinality of the residue field $A/vA$. If we set $$B_{w,v}:=\min_{n\geq 0}\{q_v^n-n\cdot w\},$$ then we have
        $$\ord_v(\zeta_A(\fs)_v)\geq B_{\wt(\fs),v}-\frac{\wt(\fs)-\dep(\fs)-\height(\fs)}{q_v-1}.$$
        In particular, $$\zeta_A(\fs)_v\in A_v~\mbox{ if }q_v\geq\wt(\fs).$$
    \end{theorem}
    We will provide an example which shows that certain conditions for the integrality of $v$-adic MZVs are necessary.
    
    In Section $5$, we formulate a new framework, the \emph{adelic} MZVs, based on our integrality result. We develop some properties of adelic MZVs and we hope that they may be helpful for the later study of the function field analogue of Theorem~\ref{upperbound of FMZV}.

\section{Preliminaries}
\subsection{Notations}

    \begin{itemize}
		\setlength{\leftskip}{0.8cm}
        \setlength{\baselineskip}{18pt}
		\item[$\mathbb{F}_q$] :=  A finite field with $q$ elements, for $q$ a power of a prime number $p$.
		\item[$A$] :=  $\mathbb{F}_q[\theta]$, the polynomial ring in the variable $\theta$ over $\mathbb{F}_q$.
		\item[$v$] := A monic, irreducible polynomial in $A$.
		\item[$\epsilon_v$] := deg$_\theta(v)$, the degree of $v$ with respect to $\theta$.
		\item[$q_v$] := $q^{\epsilon_v}$, the cardinality of the residue field $A/vA$.
		\item[$k$] :=  $\mathbb{F}_q(\theta)$, the fraction field of $A$.
		\item[$|\cdot|_v$] :=  The normalized absolute value on $k$ such that $|v|_v=q_v^{-1}$.
		\item[$|\cdot|_\infty$] :=  The normalized absolute value on $k$ such that $|\theta|_\infty=q$.
		\item[ord$_v(\cdot)$] := The associated valuation of $|\cdot|_v$.
		\item[ord$_\infty(\cdot)$] := The associated valuation of $|\cdot|_\infty$.
		\item[$k_v$] :=  The completion of $k$ with respect to $|\cdot|_v$.
		\item[$k_\infty$] :=  The completion of $k$ with respect to $|\cdot|_\infty$.
		\item[$A_v$] :=  The valuation ring inside $k_v$.
		\item[$A_\infty$] :=  The valuation ring inside $k_\infty$.
		\item[$\mathbb{C}_v$] :=  The completion of an algebraic closure of $k_v$.
		\item[$\mathbb{C}_\infty$] :=  The completion of an algebraic closure of $k_\infty$.
		\item[$\Bar{k}$] :=  A fixed algebraic closure of $k$ with fixed embeddings into $\mathbb{C}_v$ and $\mathbb{C}_\infty$ respectively.
		\item[$\wt(\fs)$] :=  $s_1+\cdots+s_r$, the weight of $\fs:=(s_1,\cdots,s_r)\in\mathbb{N}^r$.
		\item[$\dep(\fs)$] :=  $r$, the depth of $\fs:=(s_1,\cdots,s_r)\in\mathbb{N}^r$.
	\end{itemize}

\subsection{Basic setting}
    In this section, we briefly recall some basic objects which are fundamental for the arithmetic of $A$. First, we set $D_0:=1$, $L_0:=1$ and for $i\in\mathbb{N}$, we set 
    $$[i]:=\theta^{q^i}-\theta\in A, D_i:=[i][i-1]^q\cdots[1]^{q^{i-1}}\mbox{ and } L_i:=(-1)^i[i][i-1]\cdots [1].$$ 
    Next, we recall the Carlitz factorials. For each non-negative integer $n$, we write $$n=\underset{j\geq 0}{\sum}n_jq^j,~0\leq n_j\leq q-1.$$ The Carlitz factorial is defined by $$\Gamma_{n+1}:=\underset{j}{\prod}D_j^{n_j}\in A.$$ Given an index $\fs=(s_1,\cdots,s_r)\in\mathbb{N}^r$, we further define $\Gamma_\fs:=\Gamma_{s_1}\cdots\Gamma_{s_r}$. 
    For these objects, we have the following basic but useful proposition on its $v$-adic valuation.
    \begin{proposition}\label{Basic_Estimation}
        The following assertions hold:
        \begin{enumerate}
            \item Let $i\in\mathbb{Z}_{\geq 0}$. If we write $i=\alpha\cdot \epsilon_v+\beta$, $\alpha,\beta\in\mathbb{Z}_{\geq 0}$ and $0\leq\beta<\epsilon_v$, then $$\ord_v(D_i)=q^\beta\cdot\frac{q_v^{\alpha}-1}{q_v-1}~\mbox{ and }\ord_v(L_i)=\alpha.$$
            \item Let $s\in\mathbb{N}_{\geq 2}$ and $\fs=(s_1,\dots,s_r)\in\mathbb{N}^r$. Then $\ord_v(\Gamma_1)=0$, $$\ord_v(\Gamma_s)\leq \frac{s-2}{q_v-1}~\mbox{ and }\ord_v(\Gamma_\fs)\leq\frac{\wt(\fs)-\dep(\fs)-\height(\fs)}{q_v-1}.$$
        \end{enumerate}
    \end{proposition}
    
    \begin{proof}
        The case $i=0$ is clear, we may assume that $i\geq 1$. Since $[i]$ is the product of all monic irreducible polynomials $f$ in $A$ with deg$(f)$ divides $i$ \cite[Prop.~3.1.6]{Go96}, we have that ord$_v([i])=1$ if and only if $\epsilon_v$ divides $i$. Then direct calculation shows that $$\ord_v(D_i)=\underset{j=1}{\overset{i}{\sum}}q^{i-j}\cdot\ord_v([j])=\underset{j=1}{\overset{\alpha}{\sum}}q^{i-j\cdot \epsilon_v}\cdot\ord_v([j\cdot \epsilon_v])=q^\beta\cdot\frac{q_v^{\alpha}-1}{q_v-1}$$ and $$\ord_v(L_i)=\underset{j=1}{\overset{\alpha}{\sum}}\ord_v([j\cdot \epsilon_v])=\alpha.$$ The first assertion now follows. To prove the second part, we write $n:=s-1=\sum_{j\geq 0}n_jq^j$ and $j=\alpha_j\cdot \epsilon_v+\beta_j$ for each non-negative integer $j$. Then
        \begin{align*}
            \ord_v(\Gamma_{s})&=\underset{j\geq 0}{\sum}n_j\cdot\ord_v(D_j)=\underset{j\geq 0}{\sum}n_j\cdot q^{\beta_j}\cdot\frac{q_v^{\alpha_j}-1}{q_v-1}=\frac{1}{q_v-1}\left(\underset{j\geq 0}{\sum}n_jq^j-\underset{j\geq 0}{\sum}n_jq^{\beta_j}\right)\\
            &\leq\frac{1}{q_v-1}\left(n-\underset{j\geq 0}{\sum}n_j\right)\leq\frac{n-1}{q_v-1}=\frac{s-2}{q_v-1}.
        \end{align*}
        Finally, $$\ord_v(\Gamma_\fs)=\underset{j=1}{\overset{r}{\sum}}\ord_v(\Gamma_{s_j})\leq \left(\underset{j=1}{\overset{r}{\sum}}\frac{s_i-1}{q_v-1}\right)-\frac{\height(\fs)}{q_v-1}=\frac{\mbox{wt}(\fs)-\mbox{dep}(\fs)-\height(\fs)}{q_v-1}.$$
     \end{proof}

\subsection{Anderson's $t$-modules}
    In this section, we quickly review the theory of $t$-modules introduced by Anderson~\cite{A86}. Let $\tau$ be the Frobenius $q$-th power operator $$\tau:=(x\mapsto x^q):\mathbb{C}_v\to\mathbb{C}_v.$$ This $\tau$-action naturally extends to matrices by componentwise action. Let $\mathbb{C}_v[\tau]$ be the non-commutative polynomial ring generated by $\tau$ equipped with the relation $$\tau\alpha=\alpha^q\tau\mbox{ for }\alpha\in\mathbb{C}_v.$$ For any $d$-dimensional additive algebraic group $\mathbb{G}_a^d$ defined over $\mathbb{C}_v$, one may identify the ring of $\mathbb{F}_q$-linear endomorphism of $\mathbb{G}_a^d$ with $\Mat_d(\mathbb{C}_v[\tau])$. We also define the partial differential operator $$\partial:=(\underset{i\geq0}{\sum}\alpha_i\tau^i\mapsto \alpha_0):\mbox{Mat}_d(\mathbb{C}_v[\tau])\to\mbox{Mat}_d(\mathbb{C}_v).$$ Now we are ready to give a precise definition of $t$-modules.
    \begin{definition}
        Let $d\in\mathbb{N}$. A $d$-dimensional $t$-module is a pair $G=(\mathbb{G}_a^d,\rho)$, where $$\rho:\mathbb{F}_q[t]\to \mbox{Mat}_d(\mathbb{C}_v[\tau])$$ is an $\mathbb{F}_q$-linear ring homomorphism so that $\partial\rho_t-\theta I_d$ is a nilpotent matrix.
    \end{definition}
    The \emph{exponential function} of $G$ is an $\mathbb{F}_q$-linear power series of the form $$\exp_G:=I_d+\underset{i\geq1}{\sum}Q_i\tau^i,Q_i\in\mbox{Mat}_d(\mathbb{C}_v).$$ It is the unique power series satisfying the property that $$\exp_G\circ\partial\rho_a=\rho_a\circ\exp_G\mbox{ for all }a\in \mathbb{F}_q[t].$$
    The \emph{logarithm} of $G$ denoted by $\log_G$, is defined to be the formal inverse of $\exp_G$. It is a $\mathbb{F}_q$-linear power series of the form $$\log_G:=I_d+\underset{i\geq1}{\sum}P_i\tau^i,P_i\in\mbox{Mat}_d(\mathbb{C}_v),$$ with the property that $$\log_G\circ\rho_a=\partial\rho_a\circ\log_G\mbox{ for all }a\in \mathbb{F}_q[t].$$
    The logarithm $\log_G$ is one of the most important object for our study as it has a deep connection to CMSPLs and MZVs (see \cite{AT90,CM19a,CM19b}).


\section{$v$-adic Carlitz multiple star polylogarithm}
In this section, we will follow~\cite{CM19a} closely to construct a specific $t$-module $G_{\fs,\bu}$ and an explicit special point $\mathbf{v}_{\fs,\bu}$ to define the $v$-adic Carlitz multiple star polylogarithm (CMSPL). Then we study the properties of the $k$-vector space generated by certain values of $v$-adic CMSPLs with the same weight.

\subsection{Formulation through iterated extension of Carlitz tensor powers}
Throughout this section, we fix $\fs = (s_{1}, \dots, s_{r})\in \NN^{r}$ and $\bu = (u_{1}, \dots, u_{r}) \in (\ok^{\times})^{r}$.
We define the $\fs$-th Carlitz multiple polylogarithm (CMPL) as follows (see~\cite{C14}): $$\Li_\fs(z_1,\dots,z_r):=\underset{i_1>\cdots>i_r\geq 0}{\sum}\frac{z_1^{q^{i_1}}\dots z_r^{q^{i_r}}}{L_{i_1}^{s_1}\cdots L_{i_r}^{s_r}}\in k\llbracket z_1,\cdots,z_r \rrbracket.$$ 
We also define the $\fs$-th Carlitz multiple star polylogarithm (CMSPL) as follows (see~\cite{CM19a}):
$$\Li_\fs^\star(z_1,\dots,z_r):=\underset{i_1\geq\cdots\geq i_r\geq 0}{\sum}\frac{z_1^{q^{i_1}}\cdots z_r^{q^{i_r}}}{L_{i_1}^{s_1}\cdots L_{i_r}^{s_r}}\in k\llbracket z_1,\cdots,z_r \rrbracket.$$ 
We denote by $\Li_\fs(z_1,\cdots,z_r)_v$ and $\Li_\fs^\star(z_1,\cdots,z_r)_v$ when we consider the $v$-adic convergence of those two infinite series.

For $1 \leq \ell \leq r$,
we set $d_{\ell} := s_{\ell} + \cdots + s_{r}$ and $d := d_{1} + \cdots + d_{r}$.
Let $B$ be a $d \times d$-matrix of the form

\[
\left( \begin{array}{c|c|c}
B[11] & \cdots & B[1r] \\ \hline
\vdots & & \vdots \\ \hline
B[r1] & \cdots & B[rr]
\end{array} \right)
\]
where $B[\ell m]$ is a $d_{\ell} \times d_{m}$-matrix for each $\ell$ and $m$.
We call $B[\ell m]$ is the $(\ell, m)$-th block sub-matrix of $B$.

For $1 \leq \ell \leq m \leq r$, we set

\[
N_{\ell} := \left(
\begin{array}{ccccc}
0 & 1 & 0 & \cdots & 0 \\
& 0 & 1 & \ddots & \vdots \\
& & \ddots & \ddots & 0 \\
& & & \ddots & 1 \\
& & & & 0
\end{array}
\right)
\in \Mat_{d_{\ell}}(\ok),
\]

\[
N := \left(
\begin{array}{cccc}
N_{1} & & & \\
& N_{2} & & \\
& & \ddots & \\
& & & N_{r}
\end{array}
\right)
\in \Mat_{d}(\ok),
\]

\[
E[\ell m] := \left(
\begin{array}{cccc}
0 & \cdots & \cdots & 0 \\
\vdots & \ddots & & \vdots \\
0 & & \ddots & \vdots \\
1 & 0 & \cdots & 0
\end{array}
\right)
\in \Mat_{d_{\ell} \times d_{m}}(\ok) \ \ \ (\mathrm{if} \ \ell = m),
\]

\[
E[\ell m] := \left(
\begin{array}{cccc}
0 & \cdots & \cdots & 0 \\
\vdots & \ddots & & \vdots \\
0 & & \ddots & \vdots \\
(-1)^{m-\ell} \prod_{e=\ell}^{m-1} u_{e} & 0 & \cdots & 0
\end{array}
\right)
\in \Mat_{d_{\ell} \times d_{m}}(\ok) \ \ \ (\mathrm{if} \ \ell < m),
\]

\[
E := \left(
\begin{array}{cccc}
E[11] & E[12] & \cdots & E[1r] \\
& E[22] & \ddots & \vdots \\
& & \ddots & E[r-1,r] \\
& & & E[rr]
\end{array}
\right)
\in \Mat_{d}(\ok).
\]

Also, we define $\mathbb{1}_\ell:=(\delta_{\ell,j})_{1\leq j\leq d}$ to be the vector in $\Mat_{1\times d}(\Bar{k})$ where $\delta_{\ell,j}=1$ if $\ell=j$ and $\delta_{\ell,j}=0$ otherwise. Finally, we define the $t$-module $G_{\fs, \bu} := (\GG_{a}^{d}, \rho)$ by
\begin{equation}\label{E:Explicit t-moduleCMPL}
  \rho_{t} = \theta I_{d} + N + E \tau
  \in \Mat_{d}(\ok[\tau]).
\end{equation}
Note that $G_{\fs,\bu}$ depends  only on $u_{1},\ldots,u_{r-1}$.

Let
\begin{equation}\label{E:v_s,u}
\bv_{\fs, \bu} :=
\begin{array}{rcll}
\ldelim( {15}{4pt}[] & 0 & \rdelim) {15}{4pt}[] & \rdelim\}{4}{10pt}[$d_{1}$] \\
& \vdots & & \\
& 0 & & \\
& (-1)^{r-1} u_{1} \cdots u_{r} & & \\
& 0 & & \rdelim\}{4}{10pt}[$d_{2}$] \\
& \vdots & & \\
& 0 & & \\
& (-1)^{r-2} u_{2} \cdots u_{r} & & \\
& \vdots & & \vdots \\
& 0 & & \rdelim\}{4}{10pt}[$d_{r}$] \\
& \vdots & & \\
& 0 & & \\
& u_{r} & & \\[10pt]
\end{array} \in G_{\fs,\bu}(\ok).
\end{equation}
If we denote $|M|_v:=\max\{|M_{ij}|_v\}$ for each matrix $M=(M_{ij})$ with entries in $\mathbb{C}_v$, then $|\bv_{\fs,\bu}|_v<1$ when $|u_m|_v\leq 1$ for each $1\leq m<r$ and $|u_r|_v<1$. In this case, $\log_{G_{\fs,\bu}}(\bv_{\fs,\bu})$ converges $v$-adically. Furthermore, we have the following theorem.
\begin{theorem}(\cite[Thm.~3.3.3]{CM19a})\label{logarithm_interpretation}
Let $u_{1}, \dots, u_{r} \in \ok^{\times}$ with $|u_{m}|_{v} \leq 1$ for each $1 \leq m < r$ and $|u_{r}|_{v} < 1$.
Let $G_{\fs,\bu}$ and $\bv_{\fs,\bu}$ be as above.
Then we have
\[
\log_{G_{\fs,\bu}} (\bv_{\fs,\bu}) =
\begin{array}{rcll}
\ldelim( {15}{4pt}[] & * & \rdelim) {15}{4pt}[] & \rdelim\}{4}{10pt}[$d_{1}$] \\
& \vdots & & \\
& * & & \\
& (-1)^{r-1}\Li_{(s_{r}, \dots, s_{1})}^{\star}(u_{r}, \dots, u_{1})_{v} & & \\
& * & & \rdelim\}{4}{10pt}[$d_{2}$] \\
& \vdots & & \\
& * & & \\
& (-1)^{r-2}\Li_{(s_{r}, \dots, s_{2})}^{\star}(u_{r}, \dots, u_{2})_{v} & & \\
& \vdots & & \vdots \\
& * & & \rdelim\}{4}{10pt}[$d_{r}$] \\
& \vdots & & \\
& * & & \\
& \Li^{\star}_{s_{r}}(u_{r})_{v} & & \\[10pt]
\end{array} \ \ \ \in \Lie G_{\fs,\bu}(\CC_{v}).
\]
\end{theorem}

In \cite{CM19a}, the logarithmic interpretation was used to analytically continue CMSPLs $v$-adically. In what follows, we will recall their result and give the formulation of $v$-adic CMSPLs via the logarithm of the $t$-module $G_{\fs,\bu}$.

\begin{proposition}(\cite[Prop.~4.1.1]{CM19a})\label{CM19_4.1.1}
    We set $\mathcal{O}_{\mathbb{C}_v}$ is the valuation ring inside $\mathbb{C}_v$, $\fm_v$ is the maximal ideal of $\mathcal{O}_{\mathbb{C}_v}$ and $v(t)$ denote by $v(\theta)\mid_{\theta=t}\in\mathbb{F}_q[t]$. Let $u_1,\cdots,u_r\in\ok^{\times}$ with $|u_i|_v\leq 1$ for each $1\leq i\leq r$. Let $G_{\fs,\bu}$ be the $t$-module defined in (\ref{E:Explicit t-moduleCMPL}) and $\bv_{\fs,\bu}\in G_{\fs,\bu}(\ok)$ be defined in (\ref{E:v_s,u}). Let $\ell\geq 1$ be an integer such that each image of $u_i$ in $\mathcal{O}_{\mathbb{C}_v}/\fm_v\cong \overline{\mathbb{F}_q}$ is contained in $\mathbb{F}_{q^\ell}$. We further set
    \begin{equation}\label{a(t), for analytic continuation}
        a(t):=(v(t)^{d_1\ell}-1)(v(t)^{d_2\ell}-1)\cdots(v(t)^{d_r\ell}-1).
    \end{equation}
    Then $$\rho_a(\bv_{\fs,\bu})\in G_{\fs,\bu}(\fm_v).$$ In particular, $\log_{G_{\fs,\bu}}(\rho_a(\bv_{\fs,\bu}))$ converges in $\Lie G_{\fs,\bu}(\mathbb{C}_v)$.
\end{proposition}

From the Proposition \ref{CM19_4.1.1}, we can make the following formulation.

\begin{definition}
Let $$\fs=(s_1,\dots,s_r)\in\mathbb{N}^r,~\bu=(u_1,\dots,u_r)\in\ok^\times\mbox{ with }|u_i|_v\leq 1.$$ We define $$\Tilde{\fs}=(s_r,\dots,s_1),~\Tilde{\bu}=(u_r,\dots,u_1),~\Tilde{d}_\ell:=s_{r+1-\ell}+\cdots+s_1,~1\leq\ell\leq r.$$ Consider the $t$-module $G_{\Tilde{\fs},\Tilde{\bu}}$ defined in (\ref{E:Explicit t-moduleCMPL}), the special point $\bv_{\Tilde{\fs},\Tilde{\bu}}$ defined in (\ref{E:v_s,u}), and any polynomial $a(t)\in\mathbb{F}_q[t]$ so that $$a(t)\neq 0\mbox{ and } \rho_a(\bv_{\Tilde{\fs},\Tilde{\bu}})\in G_{\Tilde{\fs},\Tilde{\bu}}(\fm_v).$$ Then we define $\Li^\star_\fs(u_1,\dots,u_r)_v$ to be the value
\begin{equation}\label{definition: CMSPLs}
    \frac{(-1)^{r-1}}{a(\theta)}\times\mbox{the}~\Tilde{d}_1\mbox{th coordinate of}~\log_{G_{\Tilde{\fs},\Tilde{\bu}}}(\rho_a(\bv_{\Tilde{\fs},\Tilde{\bu}})).
\end{equation}
\end{definition}

\begin{remark}
    The proposition above guarantees the existence of such $a(t)$. Moreover, the definition above is independent of the choice of $a(t)$ (see \cite[Remark.~4.1.3]{CM19a}).
\end{remark}

\begin{remark}
    If $\bu=(u_1,\dots,u_r)\in\Bar{k}^\times$ with $|u_1|_v<1$ and $|u_i|_v\leq 1$ for $2\leq i\leq r$, then the definition above coincide with the original power series expansion, namely $$\Li^\star_\fs(u_1,\dots,u_r)_v=\underset{i_1\geq\cdots\geq i_r\geq 0}{\sum}\frac{u_1^{q^{i_1}}\dots u_r^{q^{i_r}}}{L_{i_1}^{s_1}\cdots L_{i_r}^{s_r}}\in\mathbb{C}_v.$$
\end{remark}

\subsection{The $k$-vector space spanned by $v$-adic CMSPLs}
The aim of this subsection is to study the properties of the $k$-vector space spanned by $\Li_\fs^\star(\bu)_v$ with the same weight. More precisely, we get a lower bound of the $v$-adic valuation of $\Li_\fs^\star(\bu)_v$ in terms of $\wt(\fs)$ and $\epsilon_v$. In addition, we prove a functional equation of $v$-adic CMSPLs arising from the logarithm of $t$-module. As an application, we will prove that the $k$-vector space spanned by $\Li_\fs^\star(\bu)_v$ forms an algebra.

\begin{definition}
    Let $w\in\mathbb{Z}_{\geq 0}$, $S_{0,v}:=\{1\}$ and $$S_{w,v}:=\{\Li_\fs^\star(\bu)_v\mid r\in\mathbb{N},\fs\in\mathbb{N}^r,\wt(\fs)=w,\bu\in A^r\}\mbox{ for }w>0.$$
    We define $\mathcal{L}_{w,v}$ to be the $k$-vector space generated by $S_{w,v}$.
\end{definition}

    To study the properties of $\mathcal{L}_{w,v}$, we need some further information of $\Li_\fs^\star(\bu)_v$. For the convenience of our later use, we adopt the following setting: For any $x\in\mathbb{C}_v$, we define $$x^{(i)}:=\tau^i(x)=x^{q^i}.$$
    For any $M=(M_{ij})\in\Mat_{m\times n}(\mathbb{C}_v)$, we define $M^{(i)}=(M_{ij}^{(i)})$. Furthermore, given any two square matrices $M_1,M_2$ with the same size, we define $$\ad(M_1)^0(M_2):=M_2$$ and for non-negative integers $j$, $$\ad(M_1)^{j+1}(M_2):=M_1(\ad(M_1)^j(M_2))-(\ad(M_1)^j(M_2))M_1.$$
    \begin{proposition}\label{coefficient matrix}
        Let $$\fs:=(s_1,\dots,s_r)\in\mathbb{N}^r,~\bu:=(u_1,\dots,u_r)\in(\Bar{k}^\times)^r,$$ and let $G_{\fs,\bu}$ be the $t$-module defined in (\ref{E:Explicit t-moduleCMPL}).
        Consider $$\log_{G_{\fs,\bu}}:=\underset{i\geq 0}{\sum}P_i\tau^i,~P_0=I_d,$$ and put $$d_1\mbox{-th row of }P_i:=(Y_1^{<i>},\cdots,Y_r^{<i>})\mbox{, where }Y_m^{<i>}\in\Bar{k}^{d_m}\mbox{ for }1\leq m\leq r.$$ If we set $$Y_m^{<i>}=(y_{m,1}^{<i>},\cdots,y_{m,d_m}^{<i>}),$$ then for $1\leq j\leq d_1$ we have 
        \begin{equation}\label{y^(i)_1,j}
            y_{1,j}^{<i>}=\frac{(-[i])^{d_1-j}}{L_i^{d_1}}
        \end{equation}
        and for $m\geq 2$, $1\leq j\leq d_m$ we have 
        \begin{equation}\label{y^(i)_m,j}
            y_{m,j}^{<i>}=(-1)^{m-1}(-[i])^{d_m-j}\underset{0\leq i_1\leq\cdots\leq i_{m-1}<i}{\sum}\frac{u_1^{(i_1)}\cdots u_{m-1}^{(i_{m-1})}}{L_{i_1}^{s_1}\cdots L_{i_{m-1}}^{s_{m-1}}L_i^{d_m}}.
        \end{equation}
    \end{proposition}
    
    \begin{remark}
        This proposition is an improvement of \cite[Prop.~3.2.1]{CM19a}, using the strategy of \cite{AT90}. More precisely, Chang and Mishiba proved in \cite[Prop.~3.2.1]{CM19a} that $$y_{1,d_1}^{<i>}=\frac{1}{L_i^{d_1}}$$ and for $m\geq 2$ $$y_{m,d_m}^{<i>}=(-1)^{m-1}\underset{0\leq i_1\leq\cdots\leq i_{m-1}<i}{\sum}\frac{u_1^{(i_1)}\cdots u_{m-1}^{(i_{m-1})}}{L_{i_1}^{s_1}\cdots L_{i_{m-1}}^{s_{m-1}}L_i^{d_m}}.$$
    \end{remark}
    
    \begin{remark}
        For depth one case, ie., $\fs=s$, $\bu=u$, our $t$-module $G_{\fs,\bu}$ is the $s$-th tensor power of Carlitz module and our result coincide with \cite[Cor.~4.1.5]{Pp}.
    \end{remark}
    
    \begin{proof}
        Consider the functional equation $$\log_{G_{\fs,\bu}}\circ~\rho_t=\partial\rho_t\circ\log_{G_{\fs,\bu}}.$$ We have $$(\underset{i\geq 0}{\sum}P_i\tau^i)\cdot(\theta I_d+N+E\tau)=(\theta I_d+N)\cdot(\underset{i\geq 0}{\sum}P_i\tau^i).$$ By comparison with the coefficient matrix of $\tau^{i+1}$, we obtain $$P_{i+1}(\theta^{q^{i+1}}\cdot I_d+N)+P_iE^{(i)}=(\theta\cdot I_d+N)\cdot P_{i+1}$$ 
        Therefore, $$[i+1]P_{i+1}+P_{i+1}N-NP_{i+1}=-P_iE^{(i)}$$ and thus, $$ P_{i+1}-\frac{\ad(N)^1(P_{i+1})}{[i+1]}=-\frac{P_iE^{(i)}}{[i+1]}.$$ 
        Using this recurrence relation iteratively, we finally get $$P_{i+1}=
        -\underset{j=0}{\overset{2d_1-2}{\sum}}\frac{\ad(N)^j(P_iE^{(i)})}{[i+1]^{j+1}}.$$ This equality also use the fact that $\ad(N)^{2d_1-1}(P_{i+1})=0$ since $N^{d_1}=0$. Now, we consider the vector $\mathbb{1}_{d_1}=(\delta_{d_1,\ell})_{1\leq\ell\leq d}$. Note that $\mathbb{1}_{d_1}N=0$ and $\ad(N)^j(P_iE^{(i)})$ can be expressed as $$\ad(N)^j(P_iE^{(i)})=NB+(-1)^jP_iE^{(i)}N^j$$ for some $B\in\Mat_d(\Bar{k})$. Hence
        
        \begin{equation}\label{recurrence relation}
            \mathbb{1}_{d_1}P_{i+1}=-\underset{j=0}{\overset{2d_1-2}{\sum}}\frac{\mathbb{1}_{d_1}\ad(N)^j(P_iE^{(i)})}{[i+1]^{j+1}}=\underset{j=0}{\overset{2d_1-2}{\sum}}\frac{\mathbb{1}_{d_1}P_iE^{(i)}N^j}{(-[i+1])^{j+1}}.
        \end{equation}
        For each $i,j$, we have
        \begin{align*}
            E^{(i)}N^j&=
        \left(
            \begin{array}{cccc}
                E[11]^{(i)} & E[12]^{(i)} & \cdots & E[1r]^{(i)} \\
                & E[22]^{(i)} & \ddots & \vdots \\
                & & \ddots & E[r-1,r]^{(i)} \\
                & & & E[rr]^{(i)}
            \end{array}
        \right)
        \left(
            \begin{array}{cccc}
                N_{1}^j & & & \\
                & N_{2}^j & & \\
                & & \ddots & \\
                & & & N_{r}^j
            \end{array}
        \right)\\
        &=\left(
            \begin{array}{cccc}
                E[11]^{(i)}N_1^j & E[12]^{(i)}N_2^j & \cdots & E[1r]^{(i)}N_r^j \\
                & E[22]^{(i)}N_2^j & \ddots & \vdots \\
                & & \ddots & E[r-1,r]^{(i)}N_r^j \\
                & & & E[rr]^{(i)}N_r^j
            \end{array}
        \right).
        \end{align*}
        For $1\leq\ell\leq m\leq r$, $0\leq j\leq d_m-1$, we have 
        $$E[\ell m]^{(i)}N_m^j=
         \left(
            \begin{array}{ccccccc}
                0 & \cdots & \cdots & 0 & \cdots & \cdots & 0 \\
                \vdots &  &  & \vdots &  &  & \vdots \\
                \vdots &  &  & 0 &  &  & \vdots \\
                0 & \cdots & 0 & (-1)^{m-\ell} \prod_{e=\ell}^{m-1} u_{e}^{(i)} & 0 & \cdots & 0
            \end{array}
        \right)$$ where the only non-zero element appears in the $(d_m,j+1)$-th entry. If we denote by $$Y_m^{<i>}=(y_{m,1}^{<i>},\cdots,y_{m,d_m}^{<i>}),$$ then by comparing with both sides of (\ref{recurrence relation}), we obtain 
        $$Y_1^{<i+1>}=\left(\frac{y_{1,d_1}^{<i>}}{-[i+1]},\frac{y_{1,d_1}^{<i>}}{(-[i+1])^2},\cdots,\frac{y_{1,d_1}^{<i>}}{(-[i+1])^{d_1}}\right)$$ and for $m\geq 2$ 
        $$Y_m^{<i+1>}=\left(\frac{y_{m,d_m}^{<i>}+\underset{n=1}{\overset{m-1}{\sum}}y_{n,d_n}^{<i>}(-1)^{m-n}\underset{e=n}{\overset{m-1}{\prod}}u_e^{(i)}}{-[i+1]},\cdots,\frac{y_{m,d_m}^{<i>}+\underset{n=1}{\overset{m-1}{\sum}}y_{n,d_n}^{<i>}(-1)^{m-n}\underset{e=n}{\overset{m-1}{\prod}}u_e^{(i)}}{(-[i+1])^{d_m}}\right).$$ Consequently, we get that $$y_{m,j}^{<i>}=(-[i])^{d_m-j}y_{m,d_m}^{<i>}.$$
        The desired result now follows immediately from the relation $L_{i+1}=-[i+1]L_i$ and \cite[Prop.~3.2.1]{CM19a}.
    \end{proof}
    
    In order to prove the functional equation of $v$-adic CMSPLs, we adopt the following notation. Let $\fs=(s_1,\dots,s_r)\in\mathbb{N}^r$. Then we define $$\fs_{(i)}=(s_1+\cdots+s_i, s_{i+1},\dots, s_r)\mbox{ for }1\leq i\leq r\mbox{
    and }\fs_{(i)}=\varnothing\mbox{ for }i>r\mbox{ or }i<0.$$ For convenience, we define $\Li^\star_{\varnothing}=0$. Finally, we consider the Carlitz difference operator $\vartriangle_1$, which acts on $f\in k\llbracket z_1\dots,z_r\rrbracket$ by $$(\vartriangle_1f)(z_1\cdots,z_r):=f(\theta z_1,z_2,\dots,z_r)-\theta f(z_1,\dots,z_r).$$
    
    \begin{lemma}\label{Carlitz difference operator}
        Let $\vartriangle_1^j:=\vartriangle_1\circ\cdots\circ\vartriangle_1$ be the $j$-fold composition of the Carlitz difference operator. Then we have $$(\vartriangle_1^j f)(z_1,\dots,z_r)=\underset{\ell=0}{\overset{j}{\sum}}(-1)^\ell\binom{j}{\ell}\theta^\ell f(\theta^{j-\ell}z_1,z_2,\dots,z_r).$$
    \end{lemma}
    
    \begin{proof}
        The lemma is essentially an application of the binomial theorem. We introduce two operators $S_\theta$ and $I_\theta$. Here $S_\theta$ is the operator defined by $$(S_\theta f)(z_1,\dots,z_r):=f(\theta z_1,z_2,\dots,z_r).$$ We let $I_\theta$ be the multiplication operator defined by $$(I_\theta f)(z_1,\dots,z_r):=\theta f(z_1,\dots,z_r).$$ Note that $\vartriangle_1=S_\theta-I_\theta$ and $S_\theta\circ I_\theta=I_\theta\circ S_\theta$. Thus, we have $$\vartriangle_1^j=(S_\theta-I_\theta)^j=\underset{\ell=0}{\overset{j}{\sum}}(-1)^\ell\binom{j}{\ell}S_\theta^{j-\ell}\circ I_\theta^\ell.$$ Consequently, we obtain $$(\vartriangle_1^jf)(z_1,\dots,z_r)=\underset{\ell=0}{\overset{j}{\sum}}(-1)^\ell\binom{j}{\ell}\theta^\ell f(\theta^{j-\ell}z_1,z_2,\dots,z_r).$$
    \end{proof}
    
    A simple application of lemma \ref{Carlitz difference operator} comes from the observation that $$\vartriangle_1z^{q^i}=[i]z^{q^i}.$$ More generally, if we have an $\FF_q$-linear power series $$f(z_1,\dots,z_r)=\underset{i=0}{\overset{\infty}{\sum}}c_iz_1^{q^i}\in k\llbracket z_1\dots,z_r\rrbracket,$$ where $c_i\in k\llbracket z_2,\dots,z_r\rrbracket$. Then we have $$(\vartriangle_1^j f)(z_1,\dots,z_r)=\underset{i=0}{\overset{\infty}{\sum}}c_i[i]^jz_1^{q^i}.$$ Surprisingly, this phenomenon also occurs in the calculation of the $v$-adic CMSPLs. We formulate the statement as follows.
    
    \begin{theorem}\label{functional equation}
        Let $$\fs=(s_1,\dots,s_r)\in\mathbb{N}^r,~\bu=(u_1,\dots,u_r)\in(\ok^\times)^r\mbox{ with }|u_i|_v\leq 1.$$ We define $$\Tilde{\fs}=(s_r,\dots,s_1),~\Tilde{\bu}=(u_r,\dots,u_1),~\Tilde{d}_\ell:=s_{r+1-\ell}+\cdots+s_1,~1\leq\ell\leq r.$$ Consider the $t$-module $G_{\Tilde{\fs},\Tilde{\bu}}$ defined in (\ref{E:Explicit t-moduleCMPL}), the special point $\bv_{\Tilde{\fs},\Tilde{\bu}}$ defined in (\ref{E:v_s,u}), and any polynomial $a(t)\in\mathbb{F}_q[t]$ so that $$\rho_a(\bv_{\Tilde{\fs},\Tilde{\bu}})\in G_{\Tilde{\fs},\Tilde{\bu}}(\fm_v).$$ We further set $$\rho_a(\bv_{\Tilde{\fs},\Tilde{\bu}})=(V_1,\dots,V_r)^{\tr},\mbox{ where }V_m\in\Bar{k}^{\Tilde{d}_m}.$$ If we set $$V_m:=(V_{m,1},\dots,V_{m,\Tilde{d}_m})^{\tr},$$ then we have the following identity $$\Li_\fs^\star(\bu)_v=\frac{(-1)^{r-1}}{a(\theta)}\underset{m=1}{\overset{r}{\sum}}\underset{j=0}{\overset{\Tilde{d}_m-1}{\sum}}\underset{\ell=0}{\overset{j}{\sum}}(-1)^{j+\ell+m-1}\binom{j}{\ell}\theta^\ell\{\Li_{\fs_{({r+1-m})}}^\star(\theta^{j-\ell}V_{m,\Tilde{d}_m-j},u_{r+2-m},\dots,u_r)_v$$ $$-\Li_{\fs_{(r+2-m)}}^\star(\theta^{j-\ell}V_{m,\Tilde{d}_m-j}u_{r+2-m},u_{r+3-m},\dots,u_r)_v\}.$$
    \end{theorem}
    
    \begin{proof}
        Let $\log_{G_{\Tilde{\fs},\Tilde{\bu}}}=\underset{i\geq 0}{\sum}\Tilde{P}_i\tau^i$ and 
        $$\Tilde{d}_1\mbox{-th row of }\Tilde{P}_i:=(\Tilde{Y}_1^{<i>},\dots,\Tilde{Y}_r^{<i>})\mbox{, where }\Tilde{Y}_m^{<i>}\in\Bar{k}^{\Tilde{d}_m}\mbox{ for }1\leq m\leq r.$$
        Then we have 
        \begin{align*}
            \Li_\fs^\star(\bu)_v&=\frac{(-1)^{r-1}}{a(\theta)}~\underset{i\geq 0}{\sum}(\Tilde{Y}_1^{<i>},\dots,\Tilde{Y}_r^{<i>})\cdot (V_1^{(i)},\dots,V_r^{(i)})^{\tr}\\
            &=\underset{i\geq0}{\sum}\{(\Tilde{Y}_1^{<i>})\cdot(V_1^{(i)})^{\tr}+\cdots+(\Tilde{Y}_r^{<i>})\cdot(V_r^{(i)})^{\tr}\}.
        \end{align*}
        For $m=1$, we have
        \begin{align*}
            \underset{i\geq0}{\sum}(\Tilde{Y}_1^{<i>})\cdot(V_1^{(i)})^{\tr}&=\underset{i\geq0}{\sum}~\underset{j=1}{\overset{\Tilde{d}_1}{\sum}}\frac{(-1)^{\Tilde{d}_1-j}[i]^{\Tilde{d}_1-j}V_{1,j}^{q^i}}{L_i^{\Tilde{d}_1}}=\underset{j=0}{\overset{\Tilde{d}_1-1}{\sum}}~\underset{i\geq0}{\sum}\frac{(-1)^{j}[i]^{j}V_{1,\Tilde{d}_1-j}^{q^i}}{L_i^{\Tilde{d}_1}}\\
            &=\underset{j=0}{\overset{\Tilde{d}_1-1}{\sum}}(-1)^j(\vartriangle_1^j\Li_{\Tilde{d}_1}^\star)(V_{1,\Tilde{d}_1-j})=\underset{j=0}{\overset{\Tilde{d}_1-1}{\sum}}\underset{\ell=0}{\overset{j}{\sum}}(-1)^{j+\ell}\binom{j}{\ell}\theta^\ell\Li_{\Tilde{d}_1}^\star(\theta^{j-\ell}V_{1,\Tilde{d}_1-j})\\
            &=\underset{j=0}{\overset{\Tilde{d}_1-1}{\sum}}\underset{\ell=0}{\overset{j}{\sum}}(-1)^{j+\ell}\binom{j}{\ell}\theta^\ell\{\Li_{\fs_{(r)}}^\star(\theta^{j-\ell}V_{1,\Tilde{d}_1-j})-\Li_{\fs_{(r+1)}}^\star(\theta^{j-\ell}V_{1,\Tilde{d}_1-j})\}.
        \end{align*}
        The first equality comes from (\ref{y^(i)_1,j}) and the fourth equality follows by lemma \ref{Carlitz difference operator}.\\
        For $2\leq m\leq r$, we have 
        \begin{align*}
            \underset{i\geq0}{\sum}(\Tilde{Y}_m^{<i>})\cdot(V_m^{(i)})^{\tr}&=\underset{i\geq0}{\sum}~\underset{j=1}{\overset{\Tilde{d}_m}{\sum}}(-1)^{m-1+\Tilde{d}_m-j}\underset{0\leq i_1\leq\cdots\leq i_{m-1}<i}{\sum}~\frac{[i]^{\Tilde{d}_m-j}u_r^{(i_1)}\cdots u_{r+2-m}^{(i_{m-1})}V_{m,j}^{(i)}}{L_{i_1}^{s_r}\cdots L_{i_{m-1}}^{s_{r+2-m}}L_i^{\Tilde{d}_m}}\\
            &=\underset{i\geq0}{\sum}~\underset{j=1}{\overset{\Tilde{d}_m}{\sum}}(-1)^{m-1+\Tilde{d}_m-j}(\underset{0\leq i_1\leq\cdots\leq i_{m-1}\leq i}{\sum}-\underset{0\leq i_1\leq\cdots\leq i_{m-1}=i}{\sum})~\frac{[i]^{\Tilde{d}_m-j}u_r^{(i_1)}\cdots u_{r+2-m}^{(i_{m-1})}V_{m,j}^{(i)}}{L_{i_1}^{s_r}\cdots L_{i_{m-1}}^{s_{r+2-m}}L_i^{\Tilde{d}_m}}\\
            &=\underset{j=0}{\overset{\Tilde{d}_m-1}{\sum}}(-1)^{m+j-1}\underset{i\geq0}{\sum}(\underset{0\leq i_1\leq\cdots\leq i_{m-1}\leq i}{\sum}-\underset{0\leq i_1\leq\cdots\leq i_{m-1}=i}{\sum})~\frac{[i]^{j}u_r^{(i_1)}\cdots u_{r+2-m}^{(i_{m-1})}V_{m,\Tilde{d}_m-j}^{(i)}}{L_{i_1}^{s_r}\cdots L_{i_{m-1}}^{s_{r+2-m}}L_i^{\Tilde{d}_m}}\\
            &=\underset{j=0}{\overset{\Tilde{d}_m-1}{\sum}}(-1)^{m+j-1}\{(\vartriangle_1^j\Li_{\fs_{(r+1-m)}}^\star)(V_{m,d_m-j},u_{r+2-m},\dots,u_r)\\
            &-(\vartriangle_1^j\Li_{\fs_{(r+2-m)}}^\star)(V_{m,d_m-j}u_{r+2-m},u_{r+3-m},\dots,u_r)\}\\
            &=\underset{j=0}{\overset{\Tilde{d}_m-1}{\sum}}\underset{\ell=0}{\overset{j}{\sum}}(-1)^{j+\ell+m-1}\binom{j}{\ell}\theta^\ell\{\Li_{\fs_{({r+1-m})}}^\star(\theta^{j-\ell}V_{m,\Tilde{d}_m-j},u_{r+2-m},\dots,u_r)_v\\
            &-\Li_{\fs_{(r+2-m)}}^\star(\theta^{j-\ell}V_{m,\Tilde{d}_m-j}u_{r+2-m},u_{r+3-m},\dots,u_r)_v\}.
        \end{align*}
        The first equality comes from (\ref{y^(i)_m,j}) and the fourth equality follows by lemma \ref{Carlitz difference operator}. The desired result now follows by combining the above two cases.
    \end{proof}
    
    \begin{remark}\label{stuffle}
        In order to state the following corollary of Theorem \ref{functional equation}, it is important to point out that $\Li_\fs^\star(\bu)_v\in S_{w,v}^\star$ and $\Li^\star_{\fs'}(\bu')_v\in S_{w',v}^\star$ satisfy stuffle relations analogous to the characteristic zero case. Indeed, every $\Li_\fs^\star(\bu)_v\in S^\star_{w,v}$ has a power series expansion, namely $$\Li_\fs^\star(\bu)_v=\underset{i_1\geq\cdots\geq i_r\geq 0}{\sum}\frac{u_1^{q^{i_1}}\dots u_r^{q^{i_r}}}{L_{i_1}^{s_1}\cdots L_{i_r}^{s_r}}\in k_v.$$
        Then inclusion-exclusion principle on the set $$\{i_1\geq\cdots\geq i_r\geq 0\}$$ shows that CMSPLs can be written as $\mathbb{F}_q$-linear combinations of CMPLs. For example, $$\{i_1\geq i_2\geq 0\}=\{i_1>i_2\geq 0\}\cup \{i_1=i_2\geq 0\}.$$ It follows that $$\Li^\star_{(s_1,s_2)}(u_1,u_2)_v=\Li_{(s_1,s_2)}(u_1,u_2)+\Li_{(s_1+s_2)}(u_1\cdot u_2).$$ Then \cite[Sec.~5.2]{C14} and \cite[Prop.~5.2.3]{CM19a} now provide the stuffle relations for CMSPLs. For example, for $\fs=s_1\in\mathbb{N}$, $\fs'=s_2\in \mathbb{N}$, $\bu=u_1\in vA$, $\bu'=u_2\in vA$, we have $$\Li_{s_1}^\star(u_{1})_v\cdot\Li_{s_2}^\star(u_2)_v=\Li_{(s_1,s_2)}^\star(u_1,u_2)_v+\Li_{(s_2,s_1)}^\star(u_2,u_1)_v-\Li_{(s_1+s_2)}^\star(u_1u_2)_v.$$
    \end{remark}
    
    Now, we are ready to state a corollary for the Theorem \ref{functional equation}.
    
    \begin{corollary}\label{lower bound of CMPSL's}
        Let $w\in\mathbb{Z}_{\geq 0}$, $S_{0,v}^\star:=\{1\}$ and $$S_{w,v}^\star:=\{\Li_\fs^\star(\bu)_v\mid r\in \mathbb{N},\fs\in\mathbb{N}^r,\wt(\fs)=w,\bu\in vA\times A^{r-1}\}\mbox{ for }w>0.$$
        Recall that $\mathcal{L}_{w,v}$ is the $k$-vector space generated by $S_{w,v}$.
        Then we have the following:
        \begin{enumerate}
            \item As a $k$-vector space, $\mathcal{L}_{w,v}$ is generated by $S_{w,v}^\star$. Moreover, let $$\mathcal{L}_v=\sum_{w\in\mathbb{Z}_{\geq 0}}\mathcal{L}_{w,v}.$$ Then $\mathcal{L}_v$ forms a $k$-algebra.
            \item Let $$B_{w,v}:=\min_{n\geq 0}\{q_v^n-n\cdot w\}.$$ Then $$\ord_v(\Li_\fs^\star(\bu)_v)\geq B_{w,v}~\mbox{ for every }\Li_\fs^\star(\bu)_v\in S_{w,v}.$$ In particular, $$\Li_\fs^\star(\bu)_v\in A_v~\mbox{ if }q_v\geq \wt(\fs).$$
        \end{enumerate}
    \end{corollary}
    
    \begin{proof}
        Indeed, Theorem \ref{functional equation} asserts that every $\Li_\fs^\star(\bu)_v\in S_{w,v}$ is just the $k$-linear combination of elements in $S^\star_{w,v}$. Thus, $\mathcal{L}_{w,v}$ is generated by $S_{w,v}^\star$. 
        Consequently, the following inclusion holds by Remark \ref{stuffle}: $$\mathcal{L}_{w_1,v}\cdot\mathcal{L}_{w_2,v}\subset \mathcal{L}_{w_1+w_2,v}.$$ The first assertion now follows immediately. For the second part, one observes that $$\ord_v(\Li_\fs^\star(\bu)_v)\geq B_{w,v}~\mbox{ for every }\Li_\fs^\star(\bu)_v\in S_{w,v}^\star$$ since $$\ord_v(\frac{u_1^{q^{i_1}}\dots u_r^{q^{i_r}}}{L_{i_1}^{s_1}\cdots L_{i_r}^{s_r}})\geq q^{i_1}\cdot\ord_v(u_1)-\wt(\fs)\ord_v(L_{i_1})\geq q^{i_1}-\wt(\fs)\lfloor \frac{i_1}{\epsilon_v}\rfloor\geq B_{w,v}.$$ Here, the second inequality comes from Proposition~\ref{Basic_Estimation} and the last inequality comes from the fact that if we write $i_1=\alpha\cdot \epsilon_v+\beta$, then $$q^{i_1}-\wt(\fs)\lfloor\frac{i_1}{\epsilon_v}\rfloor\geq q_v^\alpha-\wt(\fs)\alpha.$$ On the other hand, every $\Li_\fs^\star(\bu)\in S_{w,v}$ is just the $k$-linear combination of element in $S_{w,v}^\star$. Moreover, every coefficient lies in $k\cap A_v$ since $\ord_v(a(\theta))=0$. Hence, $$\ord_v(\Li_\fs^\star(\bu)_v)\geq B_{w,v}~\mbox{ for every }\Li_\fs^\star(\bu)_v\in S_{w,v}.$$ Finally, if $q_v\geq \wt(\fs)$ then we claim that for $w=\wt(\fs)$ we have $B_{w,v}\geq 0$. To see this, consider the function $$f_{w,v}(x):=q_v^x-x\cdot w.$$ Then we have $$f_{w,v}(2)=q_v^2-2w\geq q_v^2-2q_v\geq 0.$$ Moreover, we also have  $$f'(x)=\ln q_v\cdot q_v^x-w>\frac{1}{2}\cdot q_v^x-w\geq q_v^{x-1}-w\geq 0~\mbox{ for }x\geq 2.$$ Thus $f_{w,v}(x)$ is monotonically increasing on $x\geq 2$ and hence $$B_{w,v}\geq\min\{1,q_v-w\}\geq 0.$$ In other words, $$\Li_\fs^\star(\bu)_v\in A_v\mbox{ for }q_v\geq\wt(\fs).$$ Hence we complete the proof.
    \end{proof}
    
    We provide an explicit example for Theorem~\ref{functional equation} as follows:
    
    \begin{example}
        Consider $r=1$, $v=\theta$, $u\in A$ and $s\in\mathbb{N}$. In this case, we have $a(t)=t^s-1$ and thus $$\rho_a(\mathbf{v}_{\Tilde{s},\Tilde{u}})=(\binom{s}{1}\theta u,\binom{s}{2}\theta^2 u,\dots,\binom{s}{s-1}\theta^{s-1} u,\theta^s+(u^q-u))^\tr.$$ By the functional equation, we obtain $$\Li_s^\star(u)_v=\frac{1}{\theta^s-1}\{\Li_s^\star(\theta^s u+u^q-u)_v+\underset{j=1}{\overset{s-1}{\sum}}\underset{k=0}{\overset{j}{\sum}}(-1)^{j+k}\binom{j}{k}\theta^k\Li_s^\star(\binom{s}{j}\theta^{s-k} u)_v\}.$$ In particular, if $s=p^\ell$ for some $\ell\in\mathbb{Z}_{\geq 0}$, then $$\Li_s^\star(u)_v=\frac{1}{\theta^s-1}\Li_s^\star(\theta^s u+u^q-u)_v.$$
    \end{example}
    
    \begin{remark}
        If we replace $u_1,\dots,u_r$ by $r$ independent variables $z_1,\dots,z_r$, then by considering the $t$-module $G$ defined over $A[z_1,\dots,z_r]$, the formula in Theorem~\ref{functional equation} is still valid in the formal power series ring $k\llbracket z_1,\dots,z_r\rrbracket$ (cf. \cite[Rem.~3.3.6]{CM19a}).
    \end{remark}
    
    \begin{remark}
        For depth one case, ie., $\fs=s$ and $\bu=u$, our $t$-module $G_{\fs,\bu}$ is the $s$-th tensor power of the Carlitz module and the formula in Theorem~\ref{functional equation} coincides with \cite[Rem.~7.6.2]{T04}.
    \end{remark}

\section{$v$-adic multiple zeta values}
    The aim of this section is to provide a criterion for the integrality of $v$-adic multiple zeta values (MZVs). We first recall the formulation for $v$-adic MZVs via $v$-adic CMSPLs. Then we estimate the $v$-adic valuation of $v$-adic MZVs by using Corollary~\ref{lower bound of CMPSL's}. As a consequence, we provide an explicit lower bound for the $v$-adic valuation and a precise criterion for the integrality of $v$-adic MZVs.

\subsection{Formulation through $v$-adic CMSPLs}
    To introduce the formula of $v$-adic MZVs via $v$-adic CMSPLs, we need to review the Anderson-Thakur polynomials~\cite{AT90}. Let $t$ be a variable independent from $\theta$. We set $F_0:=1$, $F_i:=\underset{j=1}{\overset{i}{\prod}}(t^{q^i}-\theta^{q^i})$. Then the Anderson-Thakur polynomials $H_n\in A[t]$ is defined by the following generating function: $$\left(1-\underset{i=0}{\overset{\infty}{\sum}}\frac{F_i}{D_i\mid_{\theta=t}}x^{q^i}\right)^{-1}=\underset{n=0}{\overset{\infty}{\sum}}\frac{H_n}{\Gamma_{n+1}\mid_{\theta=t}}x^n.$$ For each $1\leq i\leq r$, we express the Anderson-Thakur polynomial $H_{s_i-1}(t)\in A[t]$ as $$H_{s_i-1}(t)=\underset{j=0}{\overset{m_i}{\sum}}u_{ij}t^j,$$ where $u_{ij}\in A$ and $u_{im_i}\neq 0.$ Now, we define $$\mathbf{J}_\fs:=\{0,1,\cdots,m_1\}\times\cdots\times\{0,1,\cdots,m_r\}.$$
    Given $\mathbf{j}=(j_1,\cdots,j_r)\in\mathbf{J}_\fs$, we put $\bu_\mathbf{j}:=(u_{1j_1},\cdots,u_{rj_r})\in A^r$ and $a_\mathbf{j}:=t^{j_1+\cdots+j_r}\in A[t]$. To state the formulation, we further introduce the following definition.

\begin{definition}[cf.~{\cite[Def.~5.2.1]{CM19b}}]
    Let $\fs:=(s_1,\cdots,s_r)\in\mathbb{N}^r$ and $\bu\in A^r$ with $r>1$ and let $S:=\{0,1\}$. For any $\mathbf{w}:=(\mathbf{w}_1,\cdots,\mathbf{w}_{r-1})\in S^{r-1}$, we define $$\mathbf{w}(\fs):=(s_1\lambda(\mathbf{w}_1)s_2\lambda(\mathbf{w}_2)\cdots\lambda(\mathbf{w}_{r-1})s_r)$$ where $\lambda(0)=","$(comma) and $\lambda(1)="+"$(addition). We also define $$\mathbf{w}^\times(\bu):=(u_1\mu(\mathbf{w}_1)u_2\mu(\mathbf{w}_2)\cdots\mu(\mathbf{w}_{r-1})u_r)$$ where $\mu(0)=","$(comma) and $\mu(1)="\times"$(multiplication).
\end{definition}

Now, we are ready to give the formulation of $v$-adic MZVs via $v$-adic CMSPLs.

\begin{definition}
    For any index $\fs:=(s_1,\cdots,s_r)\in\mathbb{N}^r$, let the notation be same as above. We renumber the set $$\{((-1)^{r-1}a_\mathbf{j}(\theta),\mathbf{w}(\fs),\mathbf{w}^\times(\bu_\mathbf{j}))\mid\mathbf{j}\in\mathbf{J}_\fs,\mathbf{w}\in S^{r-1}\}=\{(b_\ell,\fs_\ell,\bu_\ell)\}.$$ Then we define the $v$-adic MZV $\zeta_A(\fs)_v$ to be $$\zeta_A(\fs)_v:=\frac{1}{\Gamma_\fs}\sum_\ell b_\ell\cdot(-1)^{\dep(\fs_\ell)-1}\Li^\star_{\fs_\ell}(\bu_\ell)_v\in k_v.$$
\end{definition}

\begin{remark}
    This definition, given in ~\cite[Def.~6.1.1]{CM19b}, is inspired by Furusho's $p$-adic MZVs (see \cite{F04}) and the logarithmic interpretation of $\infty$-adic MZVs (see ~\cite{AT90}, ~\cite{CM19b}). Note that we have the following identity \cite[Thm.~5.2.5]{CM19b}: $$\zeta_A(\fs)=\frac{1}{\Gamma_\fs}\sum_\ell b_\ell\cdot(-1)^{\dep(\fs_\ell)-1}\Li^\star_{\fs_\ell}(\bu_\ell)\in k_\infty.$$
\end{remark}


\subsection{Main result and examples}
    In this subsection, we prove the explicit lower bound for the $v$-adic valuation and the precise criterion for the integrality of $v$-adic MZVs. 

    \begin{theorem}\label{Main result}
        Let $\fs=(s_1,\cdots,s_r)\in\mathbb{N}^r$. If we set $$B_{w,v}:=\min_{n\geq 0}\{q_v^n-n\cdot w\},$$ then we have
        $$\ord_v(\zeta_A(\fs)_v)\geq B_{\wt(\fs),v}-\frac{\wt(\fs)-\dep(\fs)-\height(\fs)}{q_v-1}.$$
        In particular, $$\zeta_A(\fs)_v\in A_v~\mbox{ if }q_v\geq\wt(\fs).$$
    \end{theorem}
    
    \begin{remark}
        We will provide a non-integral example when the restriction in the theorem is omitted. The example was found by using the computer algebra system SageMath. The author is grateful to Yoshinori Mishiba for providing the example.
    \end{remark}
    
    
    \begin{proof}
        Since $\wt(s_\ell)=\wt(s)$ for all $\ell$ by definition, we can use Corollary~\ref{lower bound of CMPSL's} to obtain $$\ord_v(\Li_{\fs_\ell}^\star(\bu_\ell)_v)\geq B_{\wt(\fs),v}\mbox{ for all }\ell.$$ On the other hand, $b_\ell$ is just the power of $\theta$ and hence $\ord_v(b_\ell)\geq 0$. Finally, $$\ord_v(\Gamma_\fs)\leq \frac{\wt(\fs)-\dep(\fs)-\height(\fs)}{q_v-1}$$ by Proposition~\ref{Basic_Estimation}. We can conclude that
        \begin{align*}
            \ord_v(\zeta_A(\fs)_v)&=\ord_v(\frac{1}{\Gamma_\fs}\sum_\ell b_\ell\cdot(-1)^{\dep(\fs_\ell)-1}\Li^\star_{\fs_\ell}(\bu_\ell)_v)\\
            &\geq \min_{\ell}\{\ord_v(b_\ell\cdot(-1)^{\dep(s_\ell)-1}\Li_{\fs_\ell}^\star(\bu_\ell)\}-\frac{\wt(\fs)-\dep(\fs)-\height(\fs)}{q_v-1}\\
            &\geq B_{\wt(\fs),v}-\frac{\wt(\fs)-\dep(\fs)-\height(\fs)}{q_v-1}.
        \end{align*}
        In particular, if $q_v\geq\wt(\fs)$, then we have $$\frac{\wt(\fs)-\dep(\fs)-\height(\fs)}{q_v-1}<1.$$ On the other hand, the condition $q_v\geq\wt(\fs)$ guarantees that $B_{\wt(\fs),v}\geq 0$. As a result, if $q_v\geq\wt(\fs)$, then $$\ord_v(\zeta_A(\fs)_v)\geq B_{\wt(\fs),v}-\frac{\wt(\fs)-\dep(\fs)}{q_v-1}>-1.$$ The desired result now follows from the fact that $\ord(\zeta_A(\fs)_v)\in\mathbb{Z}$.
       
    \end{proof}


    Now, we give a non-integral example.

\begin{example}
    Consider $q=2$, $v=\theta$ and $\fs=(4,1)$. Then $$\Gamma_{(4,1)}=\Gamma_4\cdot\Gamma_1=\theta^2+\theta,~ H_{s_1-1}(t)=H_3(t)=t^2+t$$ $$H_{s_2-1}(t)=H_0(t)=1,~\mathbf{J}_\fs=\{(0,0),(1,0),(2,0)\}$$ and 
    $$(b_1,\bs_1,\bu_1)=(1,(4,1),(0,1)),~(b_2,\bs_2,\bu_2)=(1,(5),(0)),$$ 
    $$(b_3,\bs_3,\bu_3)=(\theta,(4,1),(1,1)),~(b_4,\bs_4,\bu_4)=(\theta,(5),(1)),$$ 
    $$(b_5,\bs_5,\bu_5)=(\theta^2,(4,1),(1,1)),~(b_6,\bs_6,\bu_6)=(\theta^2,(5),(1)).$$
    Thus, $$\zeta_A(4,1)_\theta=\Li^\star_{(4,1)}(1,1)_\theta+\Li^\star_{(5)}(1)_\theta.$$ Here we use the fact that $1=-1$ since $q=2$. Note that 
    \begin{equation}\label{non-integral_example}
        \ord_\theta(d_1\mbox{-th coordinate of }P_i\tau^i(\bv))\geq 2^i-5i.
    \end{equation}
    Thus, in the computation of $\theta$-adic CMSPLs of right hand side above (\ref{non-integral_example}), we can just compute the first four terms of $\log_G(\rho_a(\bv))$. Consequently, we obtain $$\zeta_A(4,1)_\theta=\theta^{-3}+\theta^2+O(\theta^7)\not\in A_\theta.$$
\end{example}

\section{Adelic MZVs and finite MZVs}
    In this section, we formulate the adelic MZVs over function fields and investigate its properties.
    We apply \cite[Thm.~6.1.1]{CM19b} to see that the $\bar{k}$-linear space
    spanned by $\infty$-adic MZVs is isomorphic to the $\bar{k}$-linear space
    adelic MZVs. Thus, the dimension formula which was conjectured by Todd \cite{To18} also fit into our adelic framework. Finally, We discuss potential connections between adelic MZVs and finite MZVs.

\subsection{Formulation of adelic MZVs}
    We recall the $\infty$-adic MZVs over function fields which was initially studied by Thakur in \cite{T04}. Then we give our formulation of the adelic MZVs by using our integrality result Theorem~\ref{Main result}.
    \begin{definition}
        Let $\fs:=(s_1,\cdots,s_r)\in\mathbb{N}^r$. Then
        \begin{enumerate}
            \item For the infinite place, we define $$\zeta_A(\fs)_\infty:=\sum\frac{1}{a_1^{s_1}\cdots a_r^{s_r}}\in k_\infty$$ where $a_1,\cdots,a_r$ runs over all monic polynomials in $A$ with $0\leq |a_r|_\infty<\cdots<|a_1|_\infty$.
            \item For the adelic setting, we define $$\zeta_{\mathbb{A}_k}(\fs):=(\zeta_A(\fs)_v)_{v\in M_k}\in\mathbb{A}_k$$ where $M_k$ is the set of all places of $k$ and $\mathbb{A}_k$ is the adele ring of $k$.
        \end{enumerate}
    \end{definition}
    \begin{remark}
        Theorem \ref{Main result} implies that $\zeta_{\mathbb{A}_k}(\fs)$ is well-defined.
    \end{remark}

\subsection{Dimension conjecture}
    In what follows, we recall \cite[Thm.~6.4.1]{CM19b} which implies that the $\bar{k}$-linear relations among $\infty$-adic MZVs dominate the $\bar{k}$-linear relations among $v$-adic MZVs. Then we can obtain the $\bar{k}$-vector space isomorphism from above discussion.
    \begin{theorem}[{\cite[Thm.~6.4.1]{CM19b}}]\label{infinite_to_finite}
        Let $v$ be a finite place of $k$ and fix an embedding $\bar{k}\hookrightarrow \mathbb{C}_v$ Let $w$ be a positive integer and let $\overline{\mathcal{Z}}_w$ be the $\bar{k}$-vector space spanned by all $\infty$-adic MZVs of weight $w$, and let $\overline{\mathcal{Z}}_{w,v}$ be the $\bar{k}$-vector space spanned by all $v$-adic MZVs of weight $w$. Then we have a well-defined surjective $\bar{k}$-linear map $$\overline{\mathcal{Z}}_w\twoheadrightarrow\overline{\mathcal{Z}}_{w,v}$$ given by $$\zeta_A(\fs)_\infty\mapsto\zeta_A(\fs)_v$$ and kernel contains the one-dimension vector space $\bar{k}\cdot\zeta_A(w)$ when $w$ is divisible by $q-1$.
    \end{theorem}
    As an application of Theorem~\ref{Main result} and Theorem~\ref{infinite_to_finite}, we have the following corollary:
    \begin{corollary}
        Let $\overline{\mathcal{Z}}_{w,\mathbb{A}_k}$ be the $\bar{k}$-vector space spanned by all adelic MZVs of weight $w$. Then the map $$\overline{\mathcal{Z}}_w\to\overline{\mathcal{Z}}_{w,\mathbb{A}_k}$$ 
        given by $$\zeta_A(\fs)_\infty\mapsto(\zeta_A(\fs)_v)_{v\in M_k}$$ is a well-defined $\bar{k}$-linear isomorphism. In particular, $$\overline{\mathcal{Z}}_w\cong\overline{\mathcal{Z}}_{w,\mathbb{A}_k}$$ as $k$-vector space.
    \end{corollary}
    
    \begin{proof}
        Indeed, this map is well-defined surjective $k$-linear map by Theorem~\ref{infinite_to_finite} and the injectivity is coming from the $\infty$-adic coordinate.
    \end{proof}
    
    In \cite{To18}, Todd discover some linear relations among the same weight $\infty$-adic MZVs, and makes the following conjecture.
    
    \begin{conjecture}\label{Todd's dimension conjecture}
        (Todd's dimension conjecture) Let $w\in\mathbb{N}$ and $\mathcal{Z}_w$ be the $k$-vector space generated by all $\infty$-adic MZVs of weight $w$. Then we have $$\dim_k\mathcal{Z}_w = 
        \left\{
        \begin{array}{cc}
              2^{w-1} & \mbox{if }1\leq w<q, \\
              2^{w-1}-1 & \mbox{if }w=q, \\
              \sum_{~i=1}^{~q}\dim_k\mathcal{Z}_{w-i} & \mbox{if }w>q.
        \end{array}
        \right.
        $$
    \end{conjecture}
    
    In order to see the connection between the Todd's dimension conjecture and the dimension of our adelic MZVs, we recall the following result from \cite{C14}.
    
    \begin{theorem}
        \cite[Thm.~2.2.1]{C14} If the given $\infty$-adic MZVs $Z_1,\cdots,Z_m$ are linearly independent over $k$, then $$1,Z_1,\cdots,Z_m$$ are linearly independent over $\ok$.
    \end{theorem}
    
    As a consequence, studying the $\ok$-linear relations among $\infty$-adic MZVs is equivalent to studying $k$-linear relations among $\infty$-adic MZVs.
    In other words, the dimension of the $k$-vector space $\mathcal{Z}_w$ and the dimension of the $\ok$-vector space $\overline{\mathcal{Z}}_w$ must be the same. We summarize this discussion as the following theorem.
    
    \begin{theorem}
        We adopt the same notation. Then $$\dim_k\mathcal{Z}_w=\dim_{\ok}\overline{\mathcal{Z}}_w=\dim_{\ok}\overline{\mathcal{Z}}_{w,\mathbb{A}_k}.$$
    \end{theorem}
    
\subsection{Finite MZVs in positive characteristic}
    Let $\fs=(s_1,\dots,s_r)\in\mathbb{N}^r$. Consider the $k$-algebra $\mathcal{A}_k:=(\prod_v A/v A)\otimes_A k\mathbb{k}$ where $v$ runs over all finite places $v$. Following Kaneko and Zagier, finite MZVs in positive characteristic are defined by $$\zeta_{\mathcal{A}_k}(\fs):=(\zeta_{\mathcal{A}_k}(\fs)_v)_v\in \mathcal{A}_k$$ where the $v$-component $\zeta_{\mathcal{A}_k}(\fs)_v$ is defined by $$\sum_{\deg_\theta v>\deg_\theta a_1>\cdots>\deg_\theta a_r\geq 0}\frac{1}{a_1^{s_1}\cdots a_r^{s_r}}\mbox{ mod }v.$$ Several properties of these values have already been established. For example, relations between FMZVs and finite CMPLs were developed in \cite{CM17}, several identities among FMZVs were worked out in \cite{Shi18} and non-vanishing properties were studied in \cite{ANDTR19} and \cite{PP15}.
    
    Inspired by Conjecture~\ref{Zagier's conjecture} and Theorem~\ref{upperbound of FMZV}, we raise the following question:
    \begin{question}
        Let $\mathcal{Z}_{w,\mathcal{A}_k}$ be the $k$-vector space spanned by all finite MZVs of weight $w$. Do we have a well-defined surjective $k$-linear map from $\mathcal{Z}_{w,\mathbb{A}_k}$ into $\mathcal{Z}_{w,\mathcal{A}_k}$ ? In particular, is it true that $$\dim_k\mathcal{Z}_{w,\mathcal{A}_k}\leq\dim_k\mathcal{Z}_{w,\mathbb{A}_k}?$$ If the answer to the above question is positive, then how should one formulate the dimension conjecture for $\mathcal{Z}_{w,\mathcal{A}_k}$ from Conjecture~\ref{Todd's dimension conjecture}?
    \end{question}


\bibliographystyle{alpha}

\end{document}